\RequirePackage{fix-cm}
\documentclass[smallextended]{svjour3}       
\smartqed  
\usepackage{graphicx}
\usepackage{makeidx}
\usepackage{amssymb}
\usepackage{epsfig}
\usepackage{float}
\usepackage{enumerate}
\usepackage{amsmath}
\usepackage{amsfonts}

\def\N{\mathbb{N}}

\begin{document}

\title{Numerical schemes for the optimal input flow of a supply-chain
}


\author{Ciro D'Apice         \and
        Rosanna Manzo \and
        Benedetto Piccoli
}


\institute{C. D'Apice \at
              Department of Electronic and Computer Engineering, University of Salerno, Fisciano, Salerno, Italy \\
              Tel.: +39-089-962098\\
              Fax: +39-089-964218\\
              \email{cdapice@unisa.it}           
           \and
           R. Manzo \at
              Department of Electronic and Computer Engineering, University of Salerno, Fisciano, Salerno, Italy \\
              Tel.: +39-089-964262\\
              Fax: +39-089-964218\\
              \email{rmanzo@unisa.it}           
\and
           B. Piccoli \at
              Department of Mathematical Sciences, Rutgers University, Camden, NJ, US \\
              Tel.: +1-856-2256356\\
              Fax: +1-856-225\\
              \email{piccoli@camden.rutgers.edu}           
}

\date{Received: date / Accepted: date}

\maketitle

\begin{abstract}
An innovative numerical technique is presented to adjust the inflow to a
supply chain in order to achieve a desired outflow, reducing the costs of
inventory, or the goods timing in warehouses.\newline
The supply chain is modelled by a conservation law for the density of
processed parts coupled to an ODE for the queue buffer occupancy. The
control problem is stated as the minimization of a cost functional $J$
measuring the queue size and the quadratic difference between the outflow
and the expected one. The main novelty is the extensive use of generalized
tangent vectors to a piecewise constant control, which represent time shifts
of discontinuity points.\newline
Such method allows convergence results and error estimates for an
Upwind-Euler steepest descent algorithm, which is also tested by numerical
simulations.

\end{abstract}

\section{Introduction}
\label{intro}
The mathematical modeling of industrial production, as well as the
development of techniques for simulation and optimization purposes, is of
great interest in order to reduce unwanted phenomena (bottlenecks, dead
times at queues, and so on). Depending on the scale, one can distinguish
different modelling approaches, for instance discrete (Discrete Event
Simulations, \cite{F}) or continuous (Differential Equations, \cite{A2,A3,A4,H1,H2}). The latter class includes models based on partial
differential equations (\cite{BDMP,DM,DMP2,GHK1}). For a recent review
see \cite{DGHP}.

In this paper, we focus the attention on the continuous model for supply
chains proposed by G\"{o}ttlich, Herty and Klar in \cite{GHK1}, briefly GHK
model. A supply chain consists of processors with constant processing rate
and a queue in front of each processor. The dynamics of parts on a processor
is described by a conservation law, while the evolution of the queue buffer
occupancy is given by an ordinary differential equation. The latter is
simply determined by the difference of fluxes between the preceding and
following processors. The complete model consists of a coupled PDE-ODE
system.\newline
Various optimal problems, corresponding to different types of controls, have
been analysed for the GHK model (see \cite{DMP2011,GHK2,GHR,HK,KGHK}).
Typically one may consider the input flow to the whole supply chain as a
control as well as production rates and distribution coefficients in case of
supply networks. These papers provide a number of results and, in
particular, numerical algorithms to find the optimal control. In \cite{KGHK}
two discretization techniques are compared: one consisting in first writing
the adjoint system and then discretizing, while the second consists in
inverting the order, that is first discretizing and then writing the adjoint
system. All these methods compete with Discrete Events ones for numerical
accuracy and computational times, see also \cite{DGHP}.

In this paper we focus on the optimal control problem, where the control is
given by the input flow to the supply chain and the cost functional $J$ is
the sum of time-integral of queues and quadratic distance from a preassigned
desired outflow. In \cite{DMP2011}, piecewise constant controls are
considered together with generalized tangent vectors, which represent time
shifts of discontinuities of the control. The technique of such generalized
tangent vectors was extensively used for conservation laws, see \cite{BCP},
and for the case of network models, see \cite{GP,HKP}. The main result of
\cite{DMP2011} is the existence of optimal controls.

The aim of this work is to introduce an innovative numerical approach, which
builds up on the idea of generalized tangent vectors. Let us first explain
in rough words the core idea of the approach. A good numerical method, in
theory and practice, for an optimization problem is often based on a
suitable choice of \textquotedblleft perturbations". Then one can define
critical points, according to the chosen perturbations, and numerical
algorithms stemming from such definition.\newline
We base our method on perturbations of piecewise constant controls, obtained
by time shifting the discontinuity points. The advantage of this method is
twofold: on one side generalized tangent vectors are well suited for Wave
Front Tracking (briefly WFT, see \cite{B,GP}) algorithms, which allow
theoretical estimates on convergence rate. On the other side, generalized
tangent vectors have a particularly simple evolution in time, which can be
conveniently adapted to easily implementable methods as Upwind-Euler
(briefly UE, see \cite{CPR}). Moreover, in \cite{CPR} convergence of the UE
algorithm was proved by measuring the distance with WFT solutions by
generalized tangent vectors.\newline
The discretization of the evolution of generalized tangent vectors allows
the numerical computation of the cost functional gradient. This can be
combined with classical steepest descent or more advanced Newton methods for
the optimization procedure by iterations.\newline
Now, the results which complete the picture is the following. To estimate
the gradient of the cost functional we can use alternatively WFT or UE
algorithms for the evolution of generalized tangent vectors. Even more, we
can measure the distance among the two, by again using suitably defined
tangent vectors. This, in turn, allows to provide convergence results and
error estimates.\newline
Finally simulations are performed to show results for the proposed numerical
algorithm in some case studies.

The outline of the paper is the following. In Section \ref{sec:mm}, we
describe the GHK model and introduce the optimal control problem. The WFT
algorithm to construct approximate solutions to the model is illustrated in
Section \ref{sec:wft}, together with the definition and evolution of
generalized tangent vectors. Then Section \ref{sec:nm} describes the UE
numerical algorithm for solutions to the coupled ODE-PDE system and also
the numerics for generalized tangent vectors and cost functional derivative.
Convergence results and rates for our method are then given in Section \ref%
{sec:c}. Finally some numerical tests are discussed in Section \ref{sec:s}.

\section{An optimal control problem for supply chains}

\label{sec:mm}

A supply chain consists of consecutive suppliers. Each supplier is composed
of a processor for parts assembling and construction and a queue, located in
front of the processor, for unprocessed parts. Formally we have the
following definition.

\begin{definition}
A supply chain consists of a finite sequence of consecutive processors $%
I_{j} $, $j\in \mathcal{J}=\{1,\ldots ,P\}$ and queues in front of each
processor, except the first. Thus the supply chain is given by a graph $G=(V,%
\mathcal{J})$ with arcs representing processors and vertices, in $\mathcal{V}%
=\{1,\ldots ,P-1\}$, representing queues. Each processor is parametrized by
a bounded closed interval $I_{j}=[a_{j},b_{j}]$, with $%
b_{j-1}=a_{j},j=2,...,P$.
\end{definition}

The maximal processing rate $\mu _{j}$, and the processing velocity, given
by $v_{j}=L_{j}/T_{j}$ with $T_{j}$ and $L_{j}=b_{j}-a_{j}$ the processing
time and the length of the $j$-th processor, are constant parameters for
each arc. The dynamics of the $j$-th processor is given by a conservation
law
\begin{equation}
\partial _{t}\rho _{j}\left( x,t\right) +\partial _{x}\min \left\{ \mu
_{j},v_{j}\rho _{j}\left( x,t\right) \right\} =0,\quad \forall x\in \left[
a_{j},b_{j}\right] ,t\in \mathbb{R}^{+},  \label{legge}
\end{equation}%
\begin{equation*}
\rho _{j}\left( x,0\right) =\rho _{j,0}\left( x\right) ,\quad \rho
_{j}\left( a_{j},t\right) =\frac{f_{j,inc}(t)}{v_{j}},
\end{equation*}%
where $\rho _{j}\in \lbrack 0,\rho _{j}^{max}]$ is the unknown function,
representing the density of parts, while the initial datum $\rho _{j,0}$ and
the inflow $f_{j,inc}(t)$ are to be assigned. An input profile $u(t)$ on the
left boundary $\{(a_{1},t):t\in \mathbb{R}\}$ is given for the first arc
of the supply chain. Each queue buffer occupancy is modelled as a
time-dependent function $t\rightarrow q_{j}(t)$, satisfying the following
equation:
\begin{equation}
\dot{q}_{j}(t)=f_{j-1}\left( \rho _{j-1}\left( b_{j-1},t\right) \right)
-f_{j,inc},\quad j=2,...,P,  \label{legge_coda}
\end{equation}%
where the first term is defined by the trace of $\rho _{j-1}$ (which is
assumed to be of bounded variation on the $x$ variable), while the second is
defined by:
\begin{equation}
f_{j,inc}=\left\{
\begin{tabular}{ll}
$\min \left\{ f_{j-1}\left( \rho _{j-1}\left( b_{j-1},t\right) \right) ,\mu
_{j}\right\} $ & if $q_{j}\left( t\right) =0,$ \\
$\mu _{j}$ & if $q_{j}\left( t\right) >0.$%
\end{tabular}%
\right.  \label{eq:finc}
\end{equation}%
This allows for the following interpretation: If the outgoing buffer is
empty, we process as many parts as possible but at most $\mu _{j}.$ If the
buffering queue contains parts, then we process at the maximal possible
rate, namely again $\mu _{j}.$ Finally, the supply chain model is a coupled
system of partial and ordinary differential equations given by

\begin{equation}
\left\{
\begin{array}{ll}
\partial _{t}\rho _{j}\left( x,t\right) +\partial _{x}\min \left\{ \mu
_{j},v_{j}\rho _{j}\left( x,t\right) \right\} =0 & j=1,...,P, \\
\dot{q}_{j}(t)=f_{j-1}\left( \rho _{j-1}\left( b_{j-1},t\right) \right)
-f_{j,inc}(t) & j=2,...,P, \\
q_{j}\left( 0\right) =q_{j,0} & j=2,...,P, \\
\rho _{j}\left( x,0\right) =\rho _{j,0}\left( x\right) & j=1,...,P, \\
\rho _{j}\left( a_{j},t\right) =\frac{f_{j,inc}(t)}{v_{j}} & j=1,...,P, \\
f_{1,inc}(t)=u(t) &
\end{array}
\right.  \label{eq:model}
\end{equation}%
where $f_{j,inc}$ is given by (\ref{eq:finc}), for $j=2,...,P$.

Fixed a time horizon $[0,T]$, define the cost functional:

\begin{eqnarray}
J(u)& =&\sum_{j=2}^{P}\int_{0}^{T}\alpha _{1}(t)\,q_{j}(t)dt+\int_{0}^{T}\alpha
_{2}(t)\,\left[ v_{P}\cdot \rho _{P}(b_{P},t))-\psi (t)\right] ^{2}dt \notag \\
& \doteq & J_{1}(u)+J_{2}(u),  \label{func}
\end{eqnarray}%
where $\alpha _{1}\in L^{1}\big((0,T),[0,+\infty )\big)$, $\alpha _{2}\in
\text{Lip}\big((0,T),[0,+\infty )\big)$ (the space of Lipschitz continuous
functions) are weight functions, $(\rho _{j},q_{j})$ is the solution to (\ref%
{eq:model}) for the control $u$, $v_{P}\cdot \rho _{P}(b_{P},t))$ represents
the outflow of the supply chain (assuming the density level is below
$\mu _{P}$), while $\psi (t)\in \text{Lip}\big((0,T),[0,+\infty )%
\big)$ is a pre-assigned desired outflow. Given $C>0$, we consider the
minimization problem
\begin{equation}
\min_{u\in \mathcal{U}_{C}}J(u)  \label{min}
\end{equation}%
where $\mathcal{U}_{C}=\left\{ u:[0,T]\rightarrow \left[ 0,\mu _{1}\right]
\text{; }u\text{ \ measurable},\ T.V.(u)\leq C\right\} $ (with $T.V.$
indicating the total variation).
In other words, we want to minimize the queues
length and the distance between the exiting flow and the pre-assigned flow $%
\psi (t)$, using the supply chain input $u$ as control.

In \cite{DMP2011}, the existence of an optimal control was proved for a
general problem, which includes the case (\ref{eq:model})-(\ref{min}):

\begin{theorem}
\label{th:existence} (see \cite{DMP2011}) Consider the optimal control
problem (\ref{eq:model}), (\ref{min}). If $J$ is lower semicontinuous for
the $L^{1}$ norm, then there exists an optimal control.
\end{theorem}

Our aim is now to provide a new approach to solve (\ref{eq:model})-(\ref{min}%
) numerically. The key idea is to focus on piecewise constant controls and
perturb the position of discontinuity points. The procedure corresponds to
define (generalized) tangent vectors to $u$ (in the spirit of \cite{BCP}).
We can then take advantage of the knowledge of time evolution of such
tangent vectors, developed in \cite{HKP}.\newline
More precisely, we start giving the following:

\begin{definition}
We indicate by $\widetilde{{\mathcal{U}}}\subset {\mathcal{U}}_{C}$ the set
of Piecewise Constant controls. For every $u\in \widetilde{{\mathcal{U}}}$
we indicate by $\tau _{k}=\tau _{k}(u)$, $k=1,\ldots ,\delta (u)$, the
discontinuity points of $u$.
\end{definition}

We are now ready to define the perturbation to a piecewise constant control:

\begin{definition}
Given $u\in\widetilde{{\mathcal{U}}}$, a tangent vector to $u$ is a vector $%
\xi =(\xi _{1},\ldots ,\xi_{\delta(u)})\in \mathbb{R}^{\delta(u)}$
representing shifts of discontinuities. The norm of the tangent vector is
defined as:
\begin{equation*}
\Vert \xi \Vert =\sum_{k=1}^{\delta(u)}|\xi _{k}|\cdot \left\vert u(\tau
_{k}+)-u(\tau _{k}-)\right\vert .
\end{equation*}
\end{definition}

Assume now for simplicity that $\tau _{1}>0$, $\tau _{\delta (u)}<T$, and
set $\tau _{0}=0$, $\xi _{0}=0$, $\tau _{\delta (u)+1}=T$, $\xi _{\delta
(u)+1}=0$. Then given a tangent vector $\xi $ to $u$, for every $\varepsilon
$ sufficiently small we define the infinitesimal displacement as:
\begin{equation}
u_{\varepsilon }=\sum_{k=0}^{\delta (u)}\chi _{\lbrack \tau _{k}+\varepsilon
\xi _{k},\tau _{k+1}+\varepsilon \xi _{k+1}[}u(\tau _{k}+),  \label{eq:vveps}
\end{equation}%
where $\chi $ is the indicator function. In other words $u_{\varepsilon }$
is obtained from $u$ by shifting the discontinuity points of $\varepsilon
\xi $ (see Figure \ref{Figure1}).
\begin{figure}[tbh]
\includegraphics[width=3in]{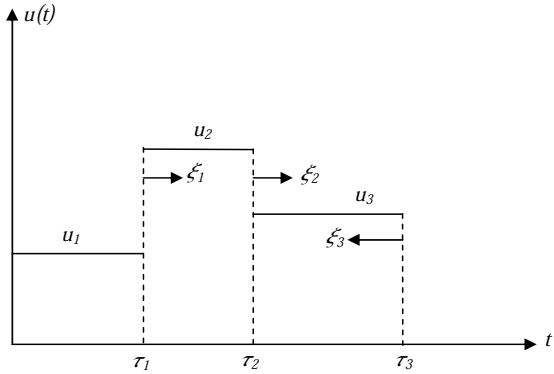}
\caption{Shifting of the discontinuity points of the control $u.$}
\label{Figure1}
\end{figure}

\begin{remark}
The notion of tangent vector to piecewise constant function was introduced
in \cite{B95} and used in \cite{BCP} to prove uniqueness and continuous
dependence of solutions to systems of conservation laws.\newline
More precisely, one defines a distance of Finsler type among piecewise
constant functions by considering paths which admit tangent vectors. Then,
by density and using the usual $L^{1}$ metric one can extend the metric to
the whole $L^{1}$.
Finally one may study the evolution of tangent vectors
and, in particular, estimates on their norms.\newline
This same technique was generalized to flows on networks, see \cite{GP}, and
to the GHK supply chain model, see \cite{HKP}.
\end{remark}

In next sections we are going to define numerical schemes for the solution
of (\ref{eq:model}) and for the evolution of tangent vectors. The latter, in
turn, will provide the information for the computation of numerical gradient
of the cost functional $J$.\newline
The evolution of tangent vectors is particularly clear for the theoretical
numerical scheme given by the WFT algorithm. This was used in
\cite{CPR} to prove convergence of an Upwind-Euler scheme and in \cite%
{DMP2011} to obtain the existence of an optimal control for (\ref{eq:model})-(%
\ref{min}). We thus recall briefly the WFT algorithm and the evolution of
tangent vectors along approximate solutions constructed via the WFT
algorithm.

\section{The Wave Front Tracking algorithm}

\label{sec:wft}

In this section we explain how to construct piecewise constant approximate
solutions to (\ref{eq:model}) by WFT method, see \cite{B} for details.

Given a discretization parameter $\sigma $ and initial conditions $\rho
_{j,0}$ in BV, the space of bounded variation functions, a WFT approximate
solution is constructed by a procedure sketched by the following steps:

\begin{itemize}
\item Approximate the initial datum by a piecewise constant function
(with discretization parameter $\sigma$) and
solve the Riemann Problems (RPs) corresponding to discontinuities of
the approximation. In RPs solutions approximate
rarefactions by rarefaction shocks of size $\sigma $;

\item Use the piecewise constant solution obtained piecing together the
solutions to RPs up to the first time of interaction of two shocks;

\item Then solve the new RP created by interaction of waves and prolong the
solution up to next interaction time, and so on.
\end{itemize}

To ensure the feasibility of such construction and the convergence of WFT
approximate solutions to a weak solution as $\sigma \rightarrow 0$, it is
enough to control the number of waves and interactions and the BV norm. This
is easily done in scalar case since both the number of waves and the BV norm
are decreasing in time (see \cite{B} for details).\newline
For our system, we need also to approximate the queue evolution. For this we
compute the exact solutions to (\ref{legge_coda}) (see \cite{HKP} for BV
estimates for the complete PDE-ODE model (\ref{eq:model})).

Notice that, as soon as a boundary datum will achieve a value below $\mu
_{j} $, then in finite time all values above $\mu _{j}$ will disappear from
the $j $-th processor, see also \cite{HKP}. Therefore, for simplicity, we
will assume

\begin{itemize}
\item[(H1)] $\rho _{j,0}(x)\leq \mu _{j}$ for all $j\in\mathcal{J}$.
\end{itemize}

Then the same inequality will be satisfied for all times. In this case
solutions to RPs are particularly simple, indeed the conservation law is
linear, thus given some Riemann data $(\rho _{-},\rho _{+})$ on the $j$-th
processor, the solution is always given by a shock travelling with velocity $%
v_{j}$.

\subsection{Tangent vectors evolution}

\label{sec:tv}

The infinitesimal displacement of each discontinuity of the control $u$
produces changes in the whole supply chain, whose effects are visible both
on processors and on queues. In fact, every shift $\xi $ generates shifts on
the densities and shifts on the queues, which spread along the whole supply
chain.\newline
Since the control $u$ is piecewise constant, the solution $(\rho _{j},q_{j})$
to (\ref{eq:model}) is such that $\rho _{j}$ is piecewise constant and $%
q_{j} $ is piecewise linear. A tangent vector to the solution $(\rho
_{j},q_{j})$ is given by:
\begin{equation*}
(^{\beta }\xi _{j},\eta _{j}),
\end{equation*}%
where $\beta$ runs over the set of discontinuities of $\rho_j$, $^{\beta
}\xi _{j}$ are the shifts of the discontinuities, while $\eta _{j}$ is the
shift of the queue buffer occupancy $q_{j} $. The norm of a tangent vector
is given by:
\begin{equation}  \label{eq:norm-tv}
\Vert (^{\beta }\xi _{j},\eta _{j})\Vert =\sum_{\beta}|^{\beta }\xi
_{j}||\Delta ^{\beta }\rho _{j}|+\sum_{j}|\eta _{j}|,
\end{equation}
where $\Delta ^{\beta }\rho _{j}=\,^\beta\rho^j_l-\,^\beta\rho^j_r$ is the
jump in $\rho$ of the discontinuity $\beta$ (where $^\beta\rho^j_l$,
respectively $^\beta\rho^j_r$, is the value on the left, respectively right,
of the discontinuity $\beta$).\newline
Notice that this is compatible with the definition of norm of tangent vector
to a control. Because of assumption (H1), we have no wave interaction inside
the processors. Therefore, densities and queues shifts remain constant for
almost all times and change only at those times at which one of the
following interactions occurs:

\begin{description}
\item[a)] \textrm{interaction of a density wave with a queue};

\item[b)] \textrm{emptying of the queue.}
\end{description}

We now provide formulas for such changes. Assume a wave with shift $%
^\beta\xi_{j-1}$ interact with the $j$-th queue and let $\bar{t}$ be the
interaction time. We use the letters $+$ and $-$ to indicate quantities
before and after $\bar{t}$, respectively. So, we indicate with $\rho
_{j}^{-} $ and $\rho _{j}^{+}$ the densities on the processor $I_{j}$ before
and after an interaction occurs and similarly for $I_{j-1}$. Also we use $%
^{\beta }\xi _{j}$ to denote the shift on the processor $I_{j}$ and with $%
^{\beta }\eta _{j}^{-}$ and $^{\beta }\eta _{j}^{+}$ the shifts on the queue
$q_{j}$, respectively before and after the interaction. Consider the case
\textbf{a)} and distinguish two subcases:

\begin{description}
\item[a.1)] $q_{j}(\bar{t})=0$;

\item[a.2)] $q_{j}(\bar{t})>0$.
\end{description}

In case \textbf{a.1)} we have to further distinguish two subcases:

\begin{description}
\item[a.1.1)] if $v_{j-1}\rho _{j-1}^{+}<v_{j-1}\rho _{j-1}^{-}<\mu _{j}$,
then $^{\beta }\xi _{j}=\frac{v_{j}}{v_{j-1}} ^\beta\xi _{j-1}$ and $^{\beta
}\eta_{j}^{-}=0=~^{\beta }\eta _{j}^{+}$ ;

\item[a.1.2)] if $v_{j-1}\rho _{j-1}^{+}>\mu _{j}$, then $^{\beta }\xi _{j}=%
\frac{v_{j}}{v_{j-1}} ^\beta\xi _{j-1}$ and $^{\beta }\eta_{j}^{+}=\
^\beta\xi_{j-1} \frac{(v_{j-1}\rho _{j-1}^{+}-\mu _{j})}{v_{j-1}} +\ ^{\beta
}\eta_{j}^{-}.$
\end{description}

In case \textbf{a.2)} we have: $^{\beta }\xi _{j}=0$, $^{\beta }\eta
_{j}^{+}=~^{\beta }\xi _{j-1}(\rho _{j-1}^{-}-\rho _{j-1}^{+})+\ ^{\beta
}\eta_{j}^{-}.$\newline
Finally in case \textbf{b)\ }we get:\textbf{\ }$^{\beta }\eta _{j}^{+}=0$, $%
^{\beta }\xi _{j-1}=0$ and $^{\beta }\xi _{j}=-\frac{v_j\ ^{\beta }\eta
_{j}^{-}}{(v_{j-1}\rho _{j-1}^{-}-\mu _{j})}.$

Using the above notations, indicate by $^{\beta }\xi _{P}$ the shift to a
generic discontinuity of $\rho _{P}$ and by $^{\beta }\rho _{P}^{+}$,
respectively $^{\beta }\rho _{P}^{-}$, the value of $\rho _{P}$ on the
right, respectively left, of the discontinuity. The following holds:

\begin{proposition}
\label{prop:Y1-Y2} Consider a control $u\in \widetilde{{\mathcal{U}}}$ and a
tangent vector $\xi \in \mathbb{R}^{\delta (u)}$ to $u$. The gradient of the
cost functional $J$ with respect to $\xi $ is given by:%
\begin{eqnarray}
\nabla _{\xi }J(u)& =&\sum_{j}\int_{0}^{T}\alpha _{1}(t)\eta
_{j}(t)dt+
\sum_{\beta }\alpha _{2}(t^{\beta })\,v_{P}\left( ^{\beta }\rho
_{P}^{+}+\,^{\beta }\rho _{P}^{-}-2\psi (t^{\beta })\right) \,^{\beta }{\xi }%
_{P}\,\Delta (\,^{\beta }\rho _{P})  \notag \\
&\doteq &Y_{1}^{WFT}+Y_{2}^{WFT},  \label{eq:Y1-Y2}
\end{eqnarray}%
where $\Delta (\,^{\beta }\rho _{P})=\,^{\beta }\rho _{P}^{+}-\,^{\beta
}\rho _{P}^{-}$ and $t^{\beta }$ is the interaction time of the
discontinuity indexed by $\beta $ with $b_{P}$, the right extreme of the
supply chain.
\end{proposition}

\begin{proof}
We have $\nabla _{\xi }J(u)=\lim_{{\varepsilon\rightarrow 0}}\frac{J(u_{\varepsilon
})-J(u)}{\varepsilon |\xi |}$. By definition of the functional $J$, the
infinitesimal change $Y_{1}^{WFT}$ in $J_{1}$ due to the infinitesimal
displacement $\varepsilon \xi $ of the control $u$ is
\begin{equation}
Y_{1}^{WFT}=\varepsilon \sum_{j}\int_{0}^{t}\alpha _{1}(t^{\beta })\eta
_{j}(t)dt,  \label{eq:J1}
\end{equation}%
while the infinitesimal change $Y_{2}^{WFT}$ in $J_{2}$ is
\begin{equation}
Y_{2}^{WFT}=\varepsilon \sum_{\beta }\alpha _{2}(t^{\beta })\,v_{P}\left[
\left( ^{\beta }\rho _{P}^{+}\right) ^{2}-\left( ^{\beta }\rho
_{P}^{-}\right) ^{2}-2\psi (t^{\beta })\left( ^{\beta }\rho
_{P}^{+}-\,^{\beta }\rho _{P}^{-}\right) \right] \,^{\beta }{\xi }_{P},
\label{eq:J3new}
\end{equation}%
thus we conclude.
\end{proof}

\section{Steepest descent for the Upwind-Euler scheme}

\label{sec:nm}

In this section we introduce first an Upwind-Euler scheme for the system (%
\ref{eq:model}) and then a numerical scheme for the evolution of the tangent
vectors to a solution to the PDE-ODE model. From the latter we will be able
to compute numerically the derivative of the cost functional with respect to
the discontinuities of the input flow. This, in turn, will be used in
steepest descent methods to find the optimal control.\newline
For a general introduction to numerical schemes for conservation laws we
refer to \cite{L}, while for optimization algorithms to \cite{Bonnans}.

For simplicity we assume:

\begin{itemize}
\item[(H2)] The lengths $L_j$ are rationally dependent.
\end{itemize}

Assumption (H2) allows us to use a unique space mesh for all processors $%
I_{j}$, $j=1,\ldots ,P$. Indeed there exists $\Delta $ so that all $L_{j}$
are multiple of $\Delta $ and we will always use time and space meshes
dividing $\Delta$.

\begin{remark}
It is possible to choose different space and/or time grid meshes for
different processors. This is necessary in the general case in which the
lengths of arcs have not rational ratios. Details can be found in \cite{CPR}.
\end{remark}

In next section we report briefly the Upwind-Euler method, analysed in \cite%
{CPR} to construct numerical solutions to the supply chain model (\ref%
{eq:model}).

\subsection{Upwind-Euler scheme for supply chains}

\label{sec:al}

Given a space mesh $\Delta x$, for each processor $I_j$, we set $\Delta
t_j=\Delta x/v_j$ and define a numerical grid of $\left[ 0,L_j\right] \times %
\left[ 0,T\right] $ by:

\begin{itemize}
\item $\left( x_{i},t^{n}\right) _{j}=\left( i\Delta x,\text{ }n\Delta
t_j\right) $, $i=0,...,N_{j},$ $n=0,...,M_{j}$ are the grid points;

\item $^{j}\rho _{i}^{n}$ is the value taken by the approximated density at
the point $\left( x_{i},t^{n}\right)_j$;

\item $q_{j}^{n}$ is the value taken by the approximate queue buffer
occupancy at time $t^n$.
\end{itemize}

The Upwind method reads:
\begin{equation}
\text{ }^{j}\rho _{i}^{n+1}=\text{ }^{j}\rho _{i}^{n}-\frac{\Delta t_{j}}{%
\Delta x}v_{j}\left( \text{ }^{j}\rho _{i}^{n}-\text{ }^{j}\rho
_{i-1}^{n}\right) =\text{ }^{j}\rho _{i-1}^{n},  \label{Upwind}
\end{equation}%
where $j\in \mathcal{J}$, $i=0,...,N_{j}$ and $n=0,...,M_{j}$. Notice that
the CFL condition is given by $\Delta t_{j}\leq \frac{\Delta x}{v_{j}},$ and
thus holds true. The explicit Euler method is given by:
\begin{equation}
q_{j}^{n+1}=q_{j}^{n}+\Delta t_{j}\left( f_{j-1}^{n}-f_{j,inc}^{n}\right) ,%
\text{ }j\in \mathcal{J}\backslash \left\{ 1\right\} ,\text{ }n=0,...,M_{j},
\label{Euler}
\end{equation}%
where $f_{j-1}^{n}$ needs to be defined and
\begin{equation}
f_{j,inc}^{n}=\left\{
\begin{tabular}{ll}
min$\left\{ f_{j-1}(^{j-1}\rho _{N_{j-1}}^{n}),\mu _{j}\right\} $ & $%
q_{j}^{n}\left( t\right) =0,$ \\
$\mu _{j}$ & $q_{j}^{n}\left( t\right) >0.$%
\end{tabular}%
\ \ \right.  \label{eq:GHK5-num}
\end{equation}%
Now, if $\Delta t_{j-1}\leq \Delta t_{j}$ we set:
\begin{equation}
f_{j-1}^{n}=\sum_{l=1}^{M\left( n\right) -m\left( n\right) -1}\Delta
t_{j-1}f_{j-1}(^{j-1}\rho _{N_{j-1}}^{m\left( n\right)
+l})=\sum_{l=1}^{\gamma }\Delta t_{j-1}f_{j-1}(^{j-1}\rho _{N_{j-1}}^{\gamma
n+l}),  \label{eq:fout1}
\end{equation}%
where $m\left( n\right) $ and $M\left( n\right) $ are defined as:%
\begin{equation*}
m\left( n\right) =\sup \left\{ m:m\Delta t_{j-1}\leq n\Delta t_{j}\right\} ,
\end{equation*}%
\begin{equation*}
M\left( n\right) =\inf \left\{ M:M\Delta t_{j-1}\geq \left( n+1\right)
\Delta t_{j}\right\} .
\end{equation*}%
Otherwise, that is if $\Delta t_{j-1}>\Delta t_{j}$, we set:
\begin{equation}
f_{j-1}^{n}=f_{j-1}\left( \,^{j-1}\rho _{N_{j-1}}^{\lfloor \frac{n\Delta
t_{j}}{\Delta t_{j-1}}\rfloor }\right) ,  \label{eq:fout2}
\end{equation}%
where \ $\lfloor \cdot \rfloor $ indicates the floor function. Boundary data
are treated using ghost cells and the expression of inflows given by (\ref%
{eq:GHK5-num}). The convergence of the scheme has been proved in \cite{CPR}\
using a comparison with WFT approximate solutions.

\subsection{Numerics for tangent vectors and cost functional}

We first completely discretize the control space via the time mesh $\Delta t$%
:
\begin{equation}
\widetilde{{\mathcal{U}}}_{\Delta t}=\{u\in\widetilde{{\mathcal{U}}}:
\tau_k(u)=n(u,k)\,\Delta t, \ n(k,u)\in\N,\ k=1,\ldots,\delta(u)\}.
\end{equation}
Now for every $u\in\widetilde{{\mathcal{U}}}_{\Delta t}$ we consider shifts $%
\xi$ so that the obtained time-shifted control is still in $\widetilde{{%
\mathcal{U}}}_{\Delta t}$. Then every $\xi_k$ is necessarily a multiple of $%
\Delta t$. Hence from now on we will restrict to the case:
\begin{equation}
\xi_k=\pm\Delta t,\quad k=1,\ldots,\delta(u).
\end{equation}

For a generic processor $I_j$ and a discontinuity point $\tau _{k}$ of the
control, we denote by $^{k,j}{\xi _{~i}^{~n}}$ and $^{k,j}{\eta ^{~n}}$ the
approximations of $^{k}\xi _{j}(x_{i},t^{n})$, and $^{k}\eta _{j}(t^{n})$,
respectively. We define such approximations by a recursive procedure
explained in the following.\newline
We initialize the tangent vector approximations by setting:
\begin{eqnarray}  \label{eq:tv-ini}
^{k,j}{\xi _{i}^{n}}~&=&0,\quad \text{for}\ n=1,\ldots ,n(k-1,u),\
j=1,\ldots ,P,  \notag \\
^{k,1}{\xi _{1}^{n(k,u)}}~&=&v_1\,(\pm\Delta t), \\
~^{k,j}{\eta^{0}}&=&0, \quad j=1,\ldots ,P.  \notag
\end{eqnarray}
The definition of $^{k,1}{\xi _{1}^{n(k,u)}}$ reflects the fact that the
shift $\xi_k$ provokes a shift of the wave generated on the first processor.%
\newline
Now, the evolution of approximations of tangent vectors to $\rho$ inside
processors is simply given by:
\begin{equation*}
^{k,j}\xi _{i}^{n+1}~=~^{k,j}\xi _{i-1}^{n}.
\end{equation*}%
On the other side, the approximation of $\xi$ and $\eta$ influence each
other at interaction times with queues. More precisely, we consider the four
cases described in Section \ref{sec:tv} and get:

\begin{description}
\item[a.1.1):] $^{k,j}{\eta }^{n+1}=0,~^{k,j}\xi _{1}^{n+1}= \frac{v_j}{%
v_{j-1}}~^{k,j-1}\xi_{N_{j-1}}^{n};$

\item[a.1.2):] $^{k,j}\xi _{1}^{n+1}=\frac{v_j}{v_{j-1}}~^{k,j-1}\xi
_{N_{j-1}}^{n}$, $^{k,j}\eta^{n+1}=~^{k,j-1}\xi _{N_{j-1}}^{n} \frac{%
\left(v_{j-1}\ ^{j-1}\rho _{N_{j-1}}^{n+1}-\mu_{j}\right)}{v_{j-1}} +\
^{k,j}\eta^{n};$

\item[a.2):] $^{k,j}\xi _{1}^{n+1}=0,~^{k,j}\eta ^{n+1}=~^{k,j-1}\xi
_{N_{j-1}}^{n}\left( ^{j-1}\rho _{N_{j-1}}^{n+1}-~^{j-1}\rho
_{N_{j-1}}^{n}\right) +~^{k,j}\eta ^{n};$

\item[b):] $^{k,j-1}\xi _{N_{j-1}}^{n}=0$, $^{k,j}\eta ^{n+1}=0$, $^{k,j}\xi
_{1}^{n+1}=-\frac{v_j\ ^{k,j}\eta ^{n}}{v_{j-1}\ ^{j-1}\rho
_{N_{j-1}}^{n}-\mu _{j}}.$
\end{description}

Notice that this is compatible with the evolution of tangent vectors along
WFT approximate solutions and also the norm of approximate tangent vectors,
in the sense of (\ref{eq:norm-tv}), is conserved.

Now we are ready to compute numerical approximations for $\nabla_\xi J$. We
denote by $^{k,j}Y_{1}^{n}$, respectively $^{k}Y_{2}^{n}$, the numerical
approximations of the $k$-th component of $Y_1^{WFT}$, respectively $%
Y_2^{WFT}$, as defined in (\ref{eq:Y1-Y2}) on processor $I_j$ at time $t^{n}$%
.\newline
We initialize such approximation by setting:
\begin{equation}
^{k,j}Y_{1}^{0}=0,\ ^{k}Y_{2}^{0}=0,\quad j=1,\ldots ,P,\ k=1,\dots
,\delta(u).
\end{equation}
Now the evolution is determined by the following simple rules. For $Y_1$, if
$q_{j}^{n+1}>0,$ then we set
\begin{equation*}
^{k,j}Y_{1}^{n+1}=~^{k,j}Y_{1}^{n}+\alpha_1(t^n)\ ^{k,j}\eta ^{n}\triangle t,
\end{equation*}
while for $q_{j}^{n+1}=0$ we distinguish two subcases:

\begin{itemize}
\item if $q_{j}^{n}=0$, then $^{k,j}Y_{1}^{n+1}=\ ^{k,j}Y_{1}^{n}$;

\item if $q_{j}^{n}>0$, then $^{k,j}Y_{1}^{n+1}=\text{ }^{k,j}Y_{1}^{n}+%
\frac{1}{2}~ \alpha_1(t^n)\,^{k,j}\xi_{1}^{n+1}~^{k,j}\eta ^{n}$.
\end{itemize}

For $Y_{2}$ we set:
\begin{equation*}
^{k}Y_{2}^{n+1}=~^{k}Y_{2}^{n}+\alpha _{2}(t)\,v_{P}\left( \left( ^{P}\rho
_{N_{P}}^{n}-\psi (t^{n})\right) ^{2}-\left( ^{P}\rho _{N_{P}}^{n+1}-\psi
(t^{n})\right) ^{2}\right) ~^{k,P}\xi _{N_{P}}^{n}.
\end{equation*}

A steepest descent algorithm, denoting whit $\vartheta $ the iteration step,
is defined by setting
\begin{equation*}
\tau _{k}^{\vartheta +1}=\tau _{k}^{\vartheta }+\lfloor \frac{h_{\theta
}\,\left( \sum_{j}\sum_{n}\text{ }^{k,j}Y_{1}^{n}+\sum_{n}\text{ }%
^{k}Y_{2}^{n}\right) }{\Delta t}\rfloor \Delta t,
\end{equation*}%
where $h_{\theta }$ is a coefficient to be suitably chosen. More precisely
the parameter $h_{\theta }$ may be chosen to solve an optimization problem
to get specific schemes, see \cite{Bonnans}.

\section{Convergence and error estimates}

\label{sec:c}

In this section we will provide convergence results and error estimates for
the Upwind-Euler steepest descent scheme illustrated in Section \ref{sec:nm}.
We will make use of two natural parameters $\nu \in \N$ and
$\theta \in \N$, the first referring to the Upwind-Euler (and WFT)
discretization, while the second indicating the iterative step of the
steepest descent algorithm.

We fix an initial space mesh $\Delta x$ (of which $\Delta$ is a multiple,
see (H2)) and define $\Delta x_{\nu }=2^{-\nu }\Delta x$. On each processor $%
I_j$ the time mesh is set as:
\begin{equation*}
\Delta t_{j,\nu }= \frac{\Delta x_{\nu }}{v_{j}},
\end{equation*}
thus granting the CFL condition. Obviously $\Delta t_{j,\nu }\rightarrow 0$,
$\Delta x_{\nu }\rightarrow 0$ as $\nu $ tends to $+\infty$.

The initial datum $\rho _{j,0}$ (see (\ref{eq:model})) is sampled in the
following way:
\begin{equation}
^{\nu ,j}\rho _{i}^{0}=\rho _{j,0}\big((a_{j}+i\,2^{-\nu }\Delta x)\,+\big).
\label{eq:s}
\end{equation}%
A control function $u_{\nu ,\theta }$ will be defined by the iteration step
of the steepest descent algorithm starting from a fixed $u_{\nu ,0}\in
\widetilde{{\mathcal{U}}}_{\Delta t_{1,\nu }}$. We will denote the
discontinuity points of $u_{\nu ,\theta }$ by $\tau _{k}^{\nu ,\theta }$,
for $k=1,\ldots ,\delta (u_{\nu ,\theta })$.\\
Notice approximations can be constructed in such a way that
$\delta(u_{\nu ,\theta })\rightarrow +\infty $ as $\nu \rightarrow +\infty $.

For simplicity $^{\nu,\theta}(\rho,q)^{UE}$ will denote the numerical
solution generated by the Upwind-Euler scheme, i.e. the collection $%
^{\nu,\theta}(^j\rho^n_i,q_j^n)$ for $j=1,\ldots ,P$, $n=1,\ldots , M^\nu_j$%
, $i=1,\ldots , N^\nu_j$. Similarly $^{\nu,\theta}(\xi,\eta)^{UE}$ will
denote the numerical tangent vectors computed as in Section \ref{sec:nm},
i.e. the collection $^{\nu,\theta}(\,^{k,j}\xi^n_i,\,^{k,j}\eta^n)$ for $%
j=1,\ldots ,P$, $n=1,\ldots , M^\nu_j$, $i=1,\ldots , N^\nu_j$, $k=1,\ldots
, \delta(u_{\nu,\theta})$.\newline
We will also use the symbols $^{\nu,\theta}(\rho,q)^{WFT}$, respectively $%
^{\nu,\theta}(\xi,\eta)^{WFT}$, to indicate the solutions, respectively
tangent vectors, produced by the Wave Front Tracking algorithm.

Now define:
\begin{equation}  \label{eq:PC}
\pi _{PC}\left( ^{\nu ,j}\rho ^{n}\right) =\sum_{i=0}^{L_{j}/2^{-\nu }\Delta
x_{j}-1}\,^{\nu ,j}\rho _{i}^{n}\,\chi _{\lbrack a_{j}+i\,2^{-\nu }\Delta
x_{j},a_{j}+(i+1)2^{-\nu }\Delta x_{j}[},
\end{equation}
and
\begin{equation*}
\pi _{PL}\left( ^{\nu }q_{j}\right) (t)=\,^{\nu }q_{j}^{n}+(t-n\Delta
t_{j,\nu })(^{\nu }q_{j}^{n+1}-\,^{\nu }q_{j}^{n})\qquad {\text{f}}\text{{or}%
}{\ }t\in \lbrack n\Delta t_{j,\nu },(n+1)\Delta t_{j,\nu }[.
\end{equation*}%
Notice that $\pi _{PC}$, respectively $\pi _{PL}$, are operators taking
values on the space of piecewise continuous, respectively piecewise linear,
functions.

In \cite{CPR} it was proved the following:

\begin{theorem}
\label{th:conv} Assume that (H1), (H2), (\ref{eq:s}) hold true and that $%
\rho _{j,0}$ are of bounded variation. Then $^{\nu ,\theta }(\rho ,q)^{WFT}$
is approximated by $^{\nu ,\theta }(\pi _{PC}(\rho ),\pi _{PL}(q))^{UE}$,
more precisely:
\begin{align}
& \Vert ^{\nu ,\theta }\rho ^{WFT}(t)-\ ^{\nu ,\theta }(\pi _{PC}(\rho
^{UE}))(t)\Vert _{L^{1}}+\Vert ^{\nu ,\theta }q^{WFT}(t)-\ ^{\nu ,\theta
}(\pi _{PL}(q^{UE}))(t)\Vert  \notag \\
& \leq 2^{-\nu }\ K\big(T.V.(\rho _{j,0}),\mu _{j}\big),  \label{eq:conv}
\end{align}
where $K$ is a suitable constant depending only on the total variation of $%
\rho _{j,0}$ and the values $\mu _{j}$, $j=1,\ldots ,P$.\newline
Moreover $^{\nu ,\theta }(\pi _{PC}(\rho ),\pi _{PL}(q))^{UE}$ converges to
a solution of (\ref{eq:model}) with a convergence rate as in (\ref{eq:conv}).
\end{theorem}

The main idea behind Theorem \ref{th:conv} is to use tangent vectors to
estimate the distance among WFT and UE approximate solutions. In the same
fashion we are now going to estimate the distance among tangent vectors
computed via WFT and UE schemes.

In each processor $I_{j}$, the WFT wave shifts $^{\beta }\xi _{j}(t^{n})$
are approximated by the UE shifts $^{k,j}\xi _{i}^{n}$ in the following
sense. Each shift $^{\beta }\xi _{j}(t^{n})$ corresponds to one or more non
vanishing values of $^{k,j}\xi _{i}^{n}$, whose position is possibly shifted
by some amount. By splitting the shift $^{\beta }\xi _{j}(t^{n})$ in one or
more pieces, whose sum is equal to $^{\beta }\xi _{j}(t^{n})$, we can define
the difference with the UE generated shifts by tangent vectors $^{\beta
}\zeta _{j}(t^{n})$, which represent the space distance among the location
of the discontinuity $\beta $ of $\rho _{j}$ at time $t^{n}$ and the
location of the non vanishing values of $^{k,j}\xi _{i}^{n}$, i.e. $%
a_{j}+i\,\Delta x_{\nu }$.\newline
For instance in the WFT algorithm a shift $\xi _{k}$, of the time
discontinuity $\tau _{k}$, gives rise to a shift $\xi =v_{1}\xi _{k}$ on
processor $I_{1}$ at time $\tau _{k}$ of the wave generated by the
discontinuity $\tau _{k}$. Similarly the UE algorithm will generate a non
vanishing value $^{k,1}\xi _{1}^{n(k,u)}=v_{1}\xi _{k}$. Therefore, we have $%
\zeta _{1}(\tau _{k})=\zeta _{1}(n(k,u)\Delta t_{1,\nu })=0$, indeed the WFT
and UE shifts coincide.\newline
Analougsly we define the tangent vectors $\iota _{j}$ to the queue shifts $%
\eta _{j}$ generated by WFT to recover those generated by UE.\newline
We define the norm of the tangent vectors $(\zeta ,\iota )$ in the following
way:
\begin{equation}
\Vert (\zeta ,\iota )\Vert =\sum_{j}\sum_{\beta }|^{\beta }\zeta
_{j}|\,|^{\beta }\xi _{j}|\,|\Delta \rho _{j}^{\beta }|+\sum_{j}|\iota
_{j}|\,|\eta _{j}|.
\end{equation}%
In particular we have:
\begin{equation}
\Vert (\zeta ,\iota )(0)\Vert =0.  \label{eq:z0}
\end{equation}

Such tangent vectors $(\zeta ,\iota )$ evolve using the same rules of $(\xi
,\eta )$ for the Wave Front Tracking algorithm, see Section \ref{sec:tv}. In
particular their norm may increase because of interactions with queues.
Moreover, the norm may also increase because of errors produced by the use
of different time meshes among different processors, see formula (\ref%
{eq:fout1}) and (\ref{eq:fout2}). Summarizing, to estimate the increase in $%
(\zeta ,\iota )$ norm we have to consider:

\begin{itemize}
\item[i)] approximation errors occurring at interaction of waves with queues
and at time in which a queue empties;

\item[ii)] the approximation errors due to different time discretizations $%
\Delta t_{j,\nu}$.
\end{itemize}

Let us start by considering the case i) assuming there is no error
approximation because of different time meshes. Since the evolution of $%
(\zeta ,\iota )$ follows the same rules as those of $(\xi ,\eta )$ for WFT,
we can use the same estimates established in \cite{CPR}. Referring to
formulas (29), (30) and (31) of \cite{CPR}, and using the symbols $\pm $ to
indicate quantities before and after the interaction, we get:
\begin{equation}
\Vert (\zeta ,\iota )_{+}\Vert \leq \left( \max \left\{ \frac{v_{j}}{v_{j-1}}%
,1\right\} \right) \Vert (\zeta ,\iota )_{-}\Vert +v_{j}\,|\eta
_{j}^{-}|\,\Delta t_{j,\nu }\,|\Delta \rho _{j}|,  \label{eq:tepossino}
\end{equation}%
where $j$ is the queue involved in the interaction and $\Delta \rho _{j}$ is
the jump in $\rho $ of the wave exiting to processor $I_{j}$. More
precisely, the first term on the right of (\ref{eq:tepossino}) is sufficient
for the case of interaction of a wave with a queue, see also (29) and (30)
of \cite{CPR}. In case of emptying of the queue $q_{j}$, the first term
takes into account the evolution of $(\eta ,\iota )$, as described in
Section \ref{sec:tv}. The second term takes into account the additional
shifts, provoked by the fact that the WFT produces a wave in $I_{j}$ at a
time which may be not a multiple of $\Delta t_{j,\nu }$ (as it happens for
all waves of the UE algorithm), see also formula (31) of \cite{CPR}.

Regarding ii), we refer to formulas (\ref{eq:fout1}) and (\ref{eq:fout2}).
In case of $\Delta t_{j-1,\nu}>\Delta t_{j,\nu}$, then no additional error
occurs. Otherwise, the error is estimated by:
\begin{equation*}
\Vert (\zeta ,\iota )_{+}\Vert \leq \Vert (\zeta ,\iota )_{-}\Vert +v_j\,
\Delta t_{j-1,\nu }\,|\xi _{j-1}|\,|\Delta \rho _{j-1}|,
\end{equation*}%
where $\xi _{j-1}$ is the shift of the interacting wave and $\Delta
\rho_{j-1}$ is the jump in $\rho$ of the interacting wave.

Recalling (\ref{eq:z0}), from the above estimates we get the following:
\begin{equation*}
\sup_{t\in \lbrack 0,T]}\Vert (\zeta ,\iota )(t)\Vert\leq \prod_{j=2}^{P}\max \left\{ \frac{v_{j}}{v_{j-1}},1\right\}
\end{equation*}%
\begin{equation}
 \cdot
\max_{j}v_{j}\Delta t_{j,\nu }\cdot \,\sup_{t\in \lbrack 0,T]}\Vert \,^{\nu
,\theta }(\xi ,\eta )^{WFT}(t)\Vert \cdot \,\sup_{t\in \lbrack
0,T]}T.V.(\,^{\nu ,\theta }\rho ^{WFT}(t)).
\end{equation}%
Now in \cite{HKP} it was proved (Lemma 2.7) that the norm of tangent vectors
are decreasing along WFT solutions. Moreover, the tangent vectors for the
WFT solutions satisfy the same initial estimate as (\ref{eq:tv-ini}), thus
we get:
\begin{equation}\label{eq:xi-nu}
\sup_{t\in \lbrack 0,T]}\Vert \,^{\nu ,\theta }(\xi ,\eta )^{WFT}(t)\Vert
\leq \Vert \,^{\nu ,\theta }(\xi ,\eta )^{WFT}(0)\Vert =v_{1}\,\Delta
t_{1,\nu }=\Delta x_{\nu }=2^{-\nu }\Delta x.
\end{equation}%
From (2.10a) of \cite{HKP} and (\ref{eq:s}), we get:
\begin{eqnarray}
\sup_{t\in \lbrack 0,T]}T.V.(\,^{\nu ,\theta }\rho ^{WFT}(t)) &\leq
&T.V.\,(\,^{\nu ,\theta }\rho ^{WFT}(0))+\left\Vert \partial _{t}\left(
\,^{\nu ,\theta }q^{WFT}(t)\right) \right\Vert  \notag  \\
&\leq &\sum_{j}T.V.(\rho _{j,0})+\sum_{j}|\mu _{j}-\mu _{j-1}|.\label{eq:WFT-dec}
\end{eqnarray}
Since $\max_{j}v_{j}\Delta t_{j,\nu }=\Delta x_\nu=2^{-\nu}\Delta x$, we get:
\begin{theorem}
\label{th:zi} Assume that (H1), (H2), (\ref{eq:s}) hold true and that $\rho
_{j,0}$ are of bounded variation. Then the following estimate holds true:
\begin{equation*}
\sup_{t\in \lbrack 0,T]}\Vert (\zeta ,\iota )(t)\Vert
\end{equation*}
\begin{equation*}
\leq \left( 2^{-\nu
}\Delta x\right) ^{2}\ \prod_{j=2}^{P}\max \left\{ \frac{v_{j}}{v_{j-1}}%
,1\right\} \ \left( \sum_{j}T.V.(\rho _{j,0})+\sum_{j}|\mu _{j}-\mu
_{j-1}|\right) .
\end{equation*}
\end{theorem}
We are now ready to prove the following:
\begin{proposition}\label{prop:K1}
Assume that (H1), (H2), (\ref{eq:s}) hold true and that $\rho _{j,0}$ are of
bounded variation. Let $Y_{i}^{UE}$, $i=1,2$,
indicate the numerical approximation of $Y_{i}$ via the Upwind-Euler steepest descent
scheme of Section \ref{sec:nm}. Then there exists a constant $K_1$ depending
only on the data of the problem:
\[
K_1=K_1\left(\Vert \alpha _{1}\Vert_{L^{1}},Lip(\alpha_2),\|\alpha_2\|_{\infty},
Lip(\psi),\|\psi\|_{\infty},v_P,\mu_j,T.V.(\rho _{j,0})\right),
\]
where $Lip(\cdot )$ indicates the Lipschitz constant of a function,
$\|\cdot \|_{\infty}$ the $L^\infty$ norm and $j=1,\ldots,P$,
such that the following estimates hold:
\begin{equation*}
\left\| Y_{1}^{WFT}-\pi _{PC}(Y_{1}^{UE})\right\| \leq K_1\
\sup_{t\in \lbrack 0,T]}\Vert (\zeta ,\iota )(t)\Vert ,
\end{equation*}%
where $\pi _{PC}$ is the projection over piecewise constant functions
as in (\ref{eq:PC}) and
\begin{equation*}
\Vert Y_{2}^{WFT}-Y_{2}^{UE}\Vert \leq K_1\
\sup_{t\in \lbrack 0,T]}\Vert (\zeta ,\iota )(t)\Vert .
\end{equation*}%
\end{proposition}

\begin{proof}
From (\ref{eq:Y1-Y2}) we have:
\begin{equation*}
\left\| Y_{1}^{WFT}-\pi _{PC}(Y_{1}^{UE})\right\| \leq \Vert \alpha _{1}\Vert
_{L^{1}}\ \sup_{t\in \lbrack 0,T]}\Vert (\zeta ,\iota )(t)\Vert ,
\end{equation*}%
and
\[
\Vert Y_{2}^{WFT}-Y_{2}^{UE}\Vert \leq
\]
\[
\left(Lip(\alpha _{2})\,v_P\,(\|\psi\|_{\infty}+2\mu_P)\,
\sup_{t\in [0,T]}\Vert \,^{\nu,\theta }(\xi ,\eta )^{WFT}(t)\Vert\,
\sup_{t\in [0,T]}T.V.(\,^{\nu ,\theta }\rho ^{WFT}(t))\right.
\]
\[
+\left.\|\alpha_2\|_{\infty}\,v_P\,Lip(\psi)\,
\sup_{t\in [0,T]}\Vert \,^{\nu,\theta }(\xi ,\eta )^{WFT}(t)\Vert\,
\sup_{t\in [0,T]}T.V.(\,^{\nu ,\theta }\rho ^{WFT}(t))\right)
\]
\[
\cdot \sup_{t\in \lbrack 0,T]}\Vert (\zeta ,\iota )(t)\Vert.
\]
By (\ref{eq:xi-nu}) and (\ref{eq:WFT-dec}) we conclude.
\end{proof}

We are now ready to state our main convergence result:
\begin{theorem}
Assume that (H1), (H2), (\ref{eq:s}) hold true and that $\rho _{j,0}$ are of
bounded variation. Fixing $\nu $, if $\delta (u_{\nu ,\theta })=\bar{\delta}$
, $h_{\theta }\geq \bar{h}>0$, for all $\theta $, and $\tau _{k}^{\nu
,\theta }\rightarrow \bar{\tau}_{k}$ as $\theta \rightarrow +\infty $, then
$u_{\nu ,\theta }$ strongly converges in $L^{1}$ to some
$\bar{u}\in {\mathcal{U}}_{C}$ as $\theta \rightarrow +\infty $ and
\begin{equation*}
\nabla _{\xi }J(\bar{u})\leq K_2\, 2^{-\nu }\,\Delta x
\end{equation*}%
where $K_2$ depends only on the data of the problem as for $K_1$
(of Proposition \ref{prop:K1}) and on $v_j$, $j=1,\ldots,P$.
\end{theorem}

\begin{proof}
Clearly $u_{\nu ,\theta }$ strongly converges in $L^{1}$ to some $\bar{u}$,
moreover by Helly theorem $\bar{u}\in {\mathcal{U}}_{C}$.\newline
Now from Proposition \ref{prop:Y1-Y2}, we have that $\nabla _{\xi }J(u_{\nu
,\theta })=Y_{1}^{WFT}+Y_{2}^{WFT}$. The tangent vectors $^{\nu ,\theta
}(\xi ,\eta )^{WFT}$ converge as $\theta \rightarrow +\infty $, in the sense
that the positions and values of shifts converge. Then we get that $\nabla
_{\xi }J(u_{\nu ,\theta })$ converges to $\nabla _{\xi }J(\bar{u})$.\newline
Now since $\tau _{k}^{\nu ,\theta }\rightarrow \bar{\tau}_{k}$ we have that
$Y_{1}^{UE}$ and $Y_{2}^{UE}$ are bounded by $\Delta t_{1,\nu }$
(times a constant depending only on the data of the problem as for $K_1$).
Thus we conclude by Theorem \ref{th:zi} and Proposition \ref{prop:K1}.
\end{proof}

\section{Simulations}
\label{sec:s}

In this Section we test the Euler-Upwind steepest descent algorithm
on two test cases.

Consider first a supply chain characterized by $11$ arcs with the following input
function:
\begin{equation*}
u\left( t\right) =\left\{
\begin{tabular}{ll}
$u_{1}$ & $0\leq t\leq \tau _{1},$ \\
$u_{2}$ & $\tau _{1}<t\leq \tau _{2},$ \\
$u_{3}$ & $\tau _{2}<t\leq T.$
\end{tabular}%
\right.
\end{equation*}%
We assume that the supply chain is initially empty and set $v_{j}$ $=1$ $\forall
j=1,...,11$ and
\begin{equation*}
\mu _{1}=200,\text{ }\mu _{2}=75,\text{ }\mu _{3}=100,\text{ }\mu _{4}=65,%
\text{ }\mu _{5}=150,
\end{equation*}%
\begin{equation*}
\mu _{6}=75,\text{ }\mu _{7}=30,\text{ }\mu _{8}=100,\text{ }\mu _{9}=80,%
\text{ }\mu _{10}=100,\text{ }\mu _{11}=120.
\end{equation*}%
For simplicity we also set $\alpha_1\equiv 1$ and $\alpha_2\equiv 0$
and analyze two different cases:
\begin{description}
\item[Case a)] $u_{1}=90,u_{2}=100,u_{3}=125,T=10,$ initial values $(\tau _{1},\tau _{2})=(1,3);$
\item[Case b)] $u_{1}=100,u_{2}=90,u_{3}=125,T=10,$ initial values $(\tau _{1},\tau _{2})=(4,5).$
\end{description}
As time and space grid meshes we choose $\Delta x=0.02$ and $\Delta t=0.016$
so as to satisfy the CFL condition. We use the condition that $J$ remains
unchanged for five runnings of the algorithm as forced stop criterion. The
times $\tau _{1}$ and $\tau _{2}$ found by the algorithm are:
\begin{description}
\item[Case a)] $\tau _{1}\simeq 8.98,$ $\tau _{2}\simeq 9.14$: as expected
both $\tau _{1}$ and $\tau _{2}$ run toward $T$; in fact, in order to
minimize the queues, the optimization algorithm tends to reduce the inflow
levels which increase the queues (i. e. $u_{2}$ and $u_{3}$).

\item[Case b)] $\tau _{1}=0,$ $\tau _{2}\simeq 9.05$: $\tau _{1}$ runs to
zero and $\tau _{2}$ runs toward $T$; as in the previous case, the
optimization algorithm works to reduce the inflow levels which lead to
queues increasing (i. e. $u_{1}$ and $u_{3}$).
\end{description}
In Table \ref{Table1} we report the numerical values of $\tau _{1}$,
$\tau _{2}$ and  $J$ at each iteration of the
steepest descent algorithm for Case a).

\begin{table}[tbph]
\caption{Case $a$: $J$ versus $\protect\tau _{1}$, $\protect\tau _{2}$
in $42$ iterations of the steepest descent algorithm.}
\label{Table1}
\begin{tabular}[t]{llll}
\hline\noalign{\smallskip}
Iteration & $\tau _{1}$ & $\tau _{2}$ & $J_{1}$ \\
1 & 1 & 3 & 117799.059 \\
2 & 2.595 & 5.988 & 89474.759 \\
3 & 3.870 & 7.480 & 79616.859 \\
4 & 4.893 & 8.228 & 75291.519 \\
5 & 5.710 & 8.600 & 73053.659 \\
6 & 6.365 & 8.788 & 71737.0579 \\
7 & 6.888 & 8.880 & 70923.172 \\
8 & 7.306 & 8.926 & 70419.523 \\
9 & 7.640 & 8.957 & 70092.102 \\
10 & 7.907 & 8.977 & 69879.287 \\
11 & 8.120 & 8.987 & 69749.807 \\
12 & 8.291 & 8.998 & 69660.449 \\
13 & 8.427 & 9.005 & 69608.289 \\
14 & 8.538 & 9.012 & 69570.370 \\
15 & 8.625 & 9.017 & 69544.690 \\
16 & 8.694 & 9.022 & 69530.770 \\
17 & 8.748 & 9.028 & 69521.531 \\
18 & 8.790 & 9.035 & 69514.2101 \\
19 & 8.818 & 9.047 & 69509.809 \\
20 & 8.839 & 9.058 & 69507.807 \\
21 & 8.858 & 9.066 & 69506.164 \\
\noalign{\smallskip}\hline
\end{tabular}%
\begin{tabular}[t]{llll}
\hline\noalign{\smallskip}
Iteration & $\tau _{1}$ & $\tau _{2}$ & $J_{1}$ \\
22 & 8.873 & 9.076 & 69504.563 \\
23 & 8.887 & 9.084 & 69503.235 \\
24 & 8.897 & 9.093 & 69501.100 \\
25 & 8.906 & 9.099 & 69501.100 \\
26 & 8.915 & 9.105 & 69500.948 \\
27 & 8.923 & 9.110 & 69500.948 \\
28 & 8.931 & 9.115 & 69500.073 \\
29 & 8.937 & 9.119 & 69500.073 \\
30 & 8.942 & 9.124 & 69500.061 \\
31 & 8.949 & 9.1276 & 69499.327 \\
32 & 8.954 & 9.130 & 69499.327 \\
33 & 8.959 & 9.134 & 69499.327 \\
34 & 8.964 & 9.137 & 69498.722 \\
35 & 8.968 & 9.139 & 69498.722 \\
36 & 8.972 & 9.140 & 69498.722 \\
37 & 8.976 & 8.972 & 69498.245 \\
38 & 8.978 & 9.144 & 69498.245 \\
39 & 8.981 & 9.146 & 69498.245 \\
40 & 8.984 & 9.147 & 69498.245 \\
41 & 8.986 & 9.149 & 69498.245 \\
42 & 8.989 & 9.150 & 69498.245 \\
\noalign{\smallskip}\hline
\end{tabular}%
\end{table}

Figure \ref{Fig_1}, respectively \ref{Fig_2}, depicts the values assumed by $J$,
respectively $(\tau _{1},$ $\tau _{2})$,
at each iteration step for Case a).

\begin{figure}[tbph]
\includegraphics[width=3.2in]{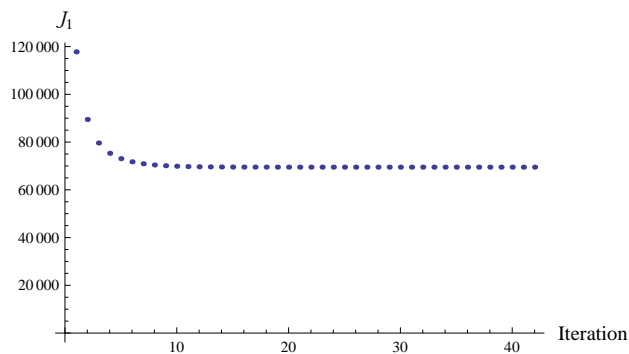}
\caption{Supply chain with $11$ arcs, Case $a$. $J$ versus iteration steps.}
\label{Fig_1}
\end{figure}

\begin{figure}[tbph]
\includegraphics[width=3in]{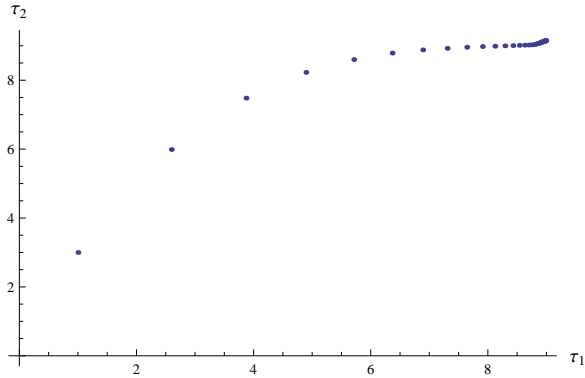}
\caption{Supply chain with $11$ arcs, Case $a$. Path followed
by the steepest descent algorithm in the plane ($\protect\tau _{1},\protect\tau _{2}$).}
\label{Fig_2}
\end{figure}
The behaviour of the cost functional $J$ in the plane $(\tau _{1},\tau _{2})$, is
reported for Case b) in Figure \ref{Fig_3}, to confirm the goodness of the
steepest descent algorithm. Notice that since $J$ decreases when the
number of iteration increases, the updating of $\tau _{1}$ and $\tau _{2}$
values allows an effective decrement of queues at supply chain nodes.
\begin{figure}[tbph]
\includegraphics[width=2.9in]{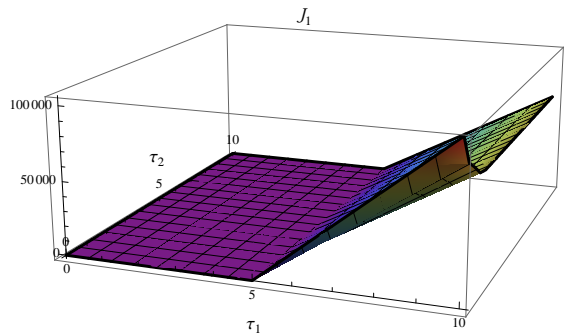}
\caption{Supply chain with $11$ arcs, Case $b$. Behaviour of $J_{1}$ in the
plane $(\protect\tau _{1},\protect\tau _{2})$. }
\label{Fig_3}
\end{figure}

We now analyze now a supply chain with $2$ arcs, maximal processing rates $\mu
_{1}=200,\mu _{2}=75$ and lengths and processing rates of each processor
equal to $1$. We assume that processors and queues are empty at $t=0$, i.e. $%
\rho _{j,0}\left( x\right) =0,$ $\forall $ $x\in \left[ 0,1\right] ,$ $j=1,2$%
, $q_{1}(0)=0$. The levels of the input flow are $u_{1}=100$, $u_{2}=80$, $%
u_{3}=50$. The total simulation time is $T=20$ and numerical approximations
are made with $\Delta x=0.02,\Delta t=0.016$. The aim is to optimize $%
J=J_{1}+J_{2},$with $\alpha _{1}\equiv\alpha _{2}\equiv0.5$, i.e. minimize the queue
handling a pre-assigned piecewise constant outflow:

\begin{equation*}
\bar{\psi}=\psi (t)=\mathbf{\ }\left\{
\begin{tabular}{ll}
$100$ & $0\leq t\leq 10,$ \\
$75$ & $10<t\leq T.$%
\end{tabular}%
\right.
\end{equation*}%
Figures \ref{Fig_4} and \ref{Fig_5} show the values assumed by $J$ at each iteration step
and the \textquotedblleft path" followed by the steepest descent algorithm
in the plane $\left( \tau _{1},\tau _{2}\right) $, starting with initial
searching point $(\tau _{1},\tau _{2})=(5,12)$. We observe that according to
the aim of minimizing the queue, $J$\ is a decreasing function and,
moreover, as expected the flux on the last arc is equal to the final outflow
value. Finally $J$ is minimized (its value is zero) by $\left(
t_{1},t_{2}\right) =\left( 0,0\right) $.

\begin{figure}[tbph]
\includegraphics[width=3.2in]{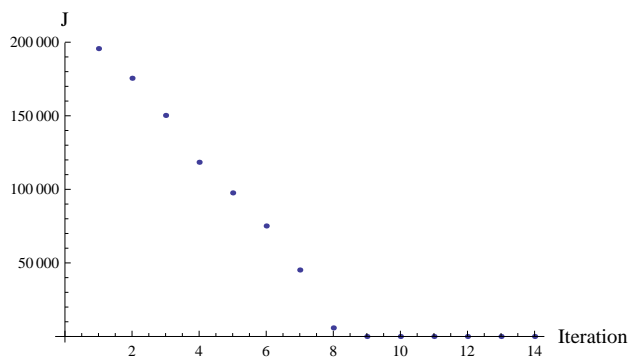} %
\caption{Supply chain with $2$ arcs: $J$ versus iteration steps.}
\label{Fig_4}
\end{figure}

\begin{figure}[tbph]
\includegraphics[width=3in]{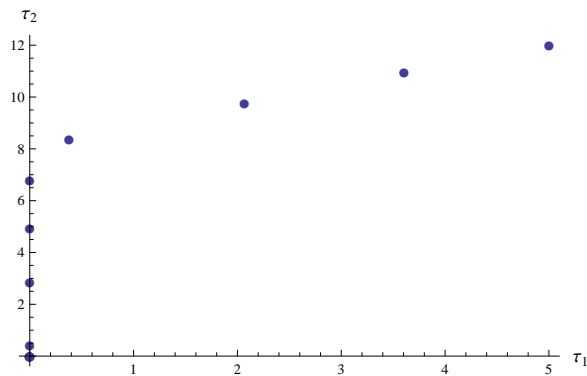}
\caption{Supply chain with $2$ arcs: \textquotedblleft path\textquotedblright\ followed by the steepest
descent algorithm in the plane $(\protect\tau _{1},\protect\tau _{2})$.}
\label{Fig_5}
\end{figure}



\begin{thebibliography}{}



\bibitem{A2} Armbruster, D., Degond, P., Ringhofer, C.: A model for the
dynamics of large queueing networks and supply chains, SIAM Journal on
Applied Mathematics, 66 (3), 896 - 920 (2006).

\bibitem{A3} Armbruster, D., Degond, P., Ringhofer, C.: Kinetic and fluid
models for supply chains supporting policy attributes, Transportation Theory
Statist. Phys., 2006b (2006).

\bibitem{A4} Armbruster, D., Marthaler, D., Ringhofer, C.: Kinetic and fluid
model hierarchies for supply chains, SIAM J. on Multiscale Modeling, 2 (1),
43 - 61 (2004).

\bibitem{Bonnans} Bonnans, J.F., Gilbert, J.C., Lemar\'{e}chal, C.,
Sagastiz\`{a}bal, C.A.: Numerical Optimization: Theoretical and Practical Aspects, 419 pages.
Springer, New York (2003).

\bibitem{B95} Bressan, A.: A variational calculus for discontinuous solutions
of systems of conservation laws, Communications on PDE, 20 (9), 1491 -
1552 (1995).

\bibitem{B} Bressan, A.: Hyperbolic Systems of Conservation Laws - The
One-dimensional Cauchy Problem, 250 pages. Oxford Univ. Press (2000).

\bibitem{BCP} Bressan, A., Crasta, G., Piccoli, B.: Well-Posedness of the
Cauchy Problem for $n\times n$ Systems of Conservation Laws, Memoirs of the
American Mathematical Society, vol. 146, n. 694 (2000).

\bibitem{BDMP} Bretti, G., D'Apice, C., Manzo, R., Piccoli, B.: A
continuum-discrete model for supply chains dynamics, Netw. Heterog. Media, 2
(4), 661 - 694 (2007).

\bibitem{CPR} Cutolo, A., Piccoli, B., Rarit\`{a}, L.: An Upwind-Euler scheme
for an ODE-PDE model of supply chains, Siam Journal on Scientific Computing,
33(4), 1669 - 1688 (2011).

\bibitem{DGHP} D'Apice, C.,  Goettlich, S., Herty, M., Piccoli, B.: Modeling,
Simulation and Optimization of Supply Chains, 216 pages. SIAM, Philadelphia PA, USA,
2010.

\bibitem{DM} D'Apice, C., Manzo, R.: A fluid dynamic model for supply chains,
Netw. Heterog. Media, 1 (3), 379 - 398 (2006).

\bibitem{DMP2} C. D'Apice, R. Manzo, B. Piccoli: Modelling supply networks
with partial differential equations, Quarterly of Applied Mathematics,
67(3), 419 - 440 (2009).

\bibitem{DMP2011} D'Apice, C., Manzo, R., Piccoli, B.: Optimal input flows for
a PDE-ODE model of supply chains, to appear on Communication in Mathematical
Sciences;

\bibitem{F} Forrester, J.W.: Industrial dynamics, 464 pages. MIT Press, Cambridge, MA (1964).

\bibitem{GP}  Garavello, M., Piccoli, B.: Traffic flows on networks, 243 pages.
American Institute of Mathematical Sciences, Springfield, MO (2006).

\bibitem{GHK1} G\"{o}ttlich, S., Herty, M., Klar, A.: Network models for
supply chains, Communication in Mathematical Sciences, 3(4), 545 - 559 (2005).

\bibitem{GHK2} G\"{o}ttlich, S., Herty, M.,  Klar, A.: Modelling and
Optimization of Supply Chains on Complex Networks, Communication in
Mathematical Sciences, 4(2), 315 - 330 (2006).

\bibitem{GHR}  G\"{o}ttlich, S., Herty, M., Ringhofer, C.: Optimization of order
policies in supply networks, European J. Oper. Res. 202(2), 456 - 465
(2010).

\bibitem{HK} Herty, M., Klar, A.: Modeling, Simulation, and Optimization of
Traffic Flow Networks, SIAM J. Sci. Comput., 25(3), 1066 - 1087 (2003).

\bibitem{HKP}  Herty, M., Klar, A., Piccoli, B.: Existence of solutions for
supply chain models based on partial differential equations, SIAM J. Math.
An., 39(1), 160 - 173 (2007).

\bibitem{H1} Helbin,g D. , L\"{a}mmer, S., Seidel, T., Seba, P., Platkowski, T.:
Physics, stability and dynamics of supply networks, Physical Review E 70, 066116 (2004).

\bibitem{H2} Helbing, D.,  L\"{a}mmer, S.: Supply and production networks:
from the bullwhip effect to business cycles, in: D. Armbruster, A. S.
Mikhailov, and K. Kaneko (eds.), Networks of Interacting Machines:
Production Organization in Complex Industrial Systems and Biological Cells,
World Scientific, Singapore, 33 - 66 (2005).

\bibitem{KGHK} Kirchner, C., Herty, M., G\"{o}ttlich, S. and Klar, A.: Optimal
control for continuous supply network models, Netw. Heterog. Media, 1(4),
675 - 688 (2006).

\bibitem{L} Leveque, R.J.: Finite Volume Methods for Hyperbolic Problems, 578 pages.
Cambridge University Press (2002).
\end{thebibliography}


\end{document}